\documentclass{amsart}
\usepackage{amssymb}
\newtheorem{thm}{Theorem}[section]

\newtheorem{lem}[thm]{Lemma}

\newtheorem{prop}[thm]{Proposition}
\theoremstyle{definition}

\theoremstyle{remark}

\numberwithin{equation}{section}

\newcommand{\R}{\mathbb R}
\newcommand{\eps}{\varepsilon}

\newcommand{\rt}{\rightarrow}

\newcommand{\rc}{\mathcal{R}}

\newcommand{\C}{{\mathbb C}}

\begin{document}

\title[almost quarter-pinching]{Almost quarter-pinched K\"ahler metrics\\ and Chern numbers}
\author{Martin Deraux}
\author{Harish Seshadri}

\address{Institut Fourier,
Universit\'{e} de Grenoble I, 38402 Saint-Martin-d'H\`{e}res Cedex, France}
\email{deraux@ujf-grenoble.fr}
\address{Department of Mathematics,
Indian Institute of Science, Bangalore 560012, India}
\email{harish@math.iisc.ernet.in}




\vspace{3mm}

\begin{abstract}
Given $n \in {\mathbb Z}^+$ and $\eps >0$, we prove that there
exists $\delta = \delta (\eps,n) >0$  such that the following
holds: If $(M^n,g)$ is a compact K\"ahler $n$-manifold whose
sectional curvatures $K$ satisfy
$$ -1 - \delta \le K \le - \frac {1}{4},$$
and $c_I(M), \ c_J(M)$ are any two Chern numbers of $M$, then
$$\Bigl \vert \frac {c_I(M)}{c_J(M)} - \frac {c_I^0}{c_J^0}
\Bigr \vert < \eps,$$ where $c_I^0, \ c_J^0$ are the corresponding
characteristic numbers of a complex-hyperbolic space-form.

It follows that the Mostow-Siu surfaces and the threefolds of Deraux do not admit K\"ahler
metrics with pinching close to $\frac {1}{4}$.
\end{abstract}


\maketitle

\section{Introduction}

It is well known that Riemannian manifolds of constant sectional
curvature $-1$ must be locally isometric to real hyperbolic
space. Allowing the sectional curvatures to be merely close to $-1$
(say with pinching close to $1$, i.e. with a suitable normalization
all sectional curvatures are between $-\alpha$ and $-1$ for some
$\alpha$ close to $1$), there are many more examples than hyperbolic
manifolds.

In fact, in~\cite{gt}, Gromov and Thurston construct manifolds of
dimension $\geq 4$ that admit negatively curved Riemannian metrics
with pinching arbitrarily close to $1$, but that do not admit any
metric of constant curvature (in fact they do not admit any locally
symmetric metric).

The K\"ahler analogue of the above examples is not well
understood. Negatively curved K\"ahler metrics must have pinching at
most $\frac{1}{4}$, as can be deduced from an inequality due to Berger
(see~\cite{ber}), and $\frac{1}{4}$-pinched metrics are precisely those
locally isometric to complex hyperbolic space (of dimension $\geq 2$).
It is not known, however, whether there exist negatively curved
K\"ahler manifolds with pinching arbitrarily close to $\frac{1}{4}$ which are not biholomorphic to ball quotients.

The surfaces constructed by Mostow and Siu in~\cite{ms} (as well as
their three-dimensional analogue, see~\cite{mar}) admit K\"ahler
metrics with negative sectional curvatures, but very little is known
about the pinching of these metrics.

The main result of this note relates proximity to $\frac{1}{4}$-pinching with
proximity of the ratios of Chern numbers to the corresponding ratios
of Chern numbers of a complex hyperbolic manifold of the same
dimension. This shows that neither the Mostow-Siu surfaces, nor the
three-dimensional analogues constructed in~\cite{mar}, admit
K\"ahler metric with pinching close to $\frac{1}{4}$. This answers a question
raised by Pansu in~\cite{pan} (echoing a question of Gromov), which
seems not to be answered anywhere in the literature.

\section{Pinching and ratios of Chern numbers}

The main observation of this note is the following:

\begin{thm}\label{thm:main}
Given $n \in {\mathbb Z}^+$ and $\eps >0$, there exists $\delta =
\delta (\eps,n) >0$  such that the following holds: If $(M^n,g)$
is a compact K\"ahler $n$-manifold whose sectional curvatures $K$
satisfy
$$ -1 - \delta \le K \le - \frac {1}{4},$$
and $c_I(M), \ c_J(M)$ are any two Chern numbers of $M$, then
$$\Bigl \vert \frac {c_I(M)}{c_J(M)} - \frac {c_I^0}{c_J^0}
\Bigr \vert < \eps,$$ where $c_I^0, \ c_J^0$ are the corresponding
characteristic numbers of a complex-hyperbolic space-form.
\end{thm}

Recall that Mostow and Siu constructed an \emph{infinite} family of
surfaces that admit negatively curved K\"ahler metrics, but the
corresponding ratios $c_1^2/c_2$ do not approach $3$ as one ranges
through these examples (see the introduction of~\cite{ms}).  The
ratios $c_1^3/c_3$ for the threefolds constructed in~\cite{mar}
certainly cannot approach $16$, since only a finite number of
threefolds were constructed there.

Theorem~\ref{thm:main} implies that the Mostow-Siu surfaces (as well
as the threefolds in~\cite{mar}) do not carry K\"ahler metrics with
pinching close to $1/4$.

It is an interesting open question whether the conclusion of Theorem ~\ref{thm:main} 
holds if only assumes the existence of almost-quarter pinched {\it Riemannian} metrics on a compact
K\"ahler manifold. This question is natural in light of the result of Hernandez-Zheng-Yau (\cite{her},~\cite{zy}): If a compact K\"ahler manifold $M$ admits a negatively quarter-pinched Riemannian $g$ metric, then $(M,g)$ is holomorphically  isometric to a complex-hyperbolic space-form.

\section{proof}
Let $V$ be a vector space over $\R$ of real dimension $2n$ with an
almost complex structure $J: V \rt V$ and a $J$ invariant inner
product $\langle \ , \ \rangle$. Hence $J^2=-I$ and $\langle J(v),J(w)
\rangle= \langle v,w \rangle$ for all $v,w \in V$.

A {\it curvature tensor} on $V$ is a 4-linear map $R: V \times V
\times V \times V \rightarrow {\mathbb R}$, i.e. an element of
$\otimes ^4 V^\ast$, satisfying \vspace{2mm}

(1) $R(X,Y,Z,W) = -R(Y,X,Z,W)= -R(X,Y,W,Z)$

(2) $R(Z,W,X,Y)=R(X,Y,Z,W)$

(3) $R(X,Y,Z,W)+ R(X,W,Y,Z)+R(X,Z,W,Y)=0$. \vspace{2mm}

If, in addition, $R$ satisfies \vspace{2mm}

(4) $R(J(X),J(Y),Z,W)=R(X,Y,J(Z),J(W))=R(X,Y,Z,W)$ \vspace{2mm}

then we say that $R$ is a {\it K\"ahler curvature tensor}.
\vspace{2mm}

Let \ $\rc \subset   \otimes^4 V^\ast$ denote the space of
K\"ahler curvature tensors and $R_0 \in \rc$ denote the curvature
tensor of complex hyperbolic space $\C H^n$.

$R_0$ is given by
\begin{align}
-4R_0(u,v,z,w) \ =& \ \langle u,z \rangle \langle v,w \rangle  -
\langle u,w \rangle \langle v,z \rangle +\langle u,J(z) \rangle
\langle v,J(w) \rangle \\ \notag
& -\langle u,J(w) \rangle \langle v,J(z)
\rangle +2\langle u,J(v) \rangle \langle z,J(w) \rangle  \notag
\end{align}
It is well-known that if  $R$ is a K\"ahler curvature tensor on
$V$ whose sectional curvatures are negative and $\frac
{1}{4}$-pinched, then $R = R_0$. We generalize this in Proposition
\ref{ta}. Recall that we regard the space $\rc$ of K\"ahler
curvature tensors as a subspace of $ \otimes^4 V^\ast$. The inner
product on $V$ canonically induces an inner product on $\otimes^4
V^\ast$ and hence on $\rc$.

In what follows, we will use the following notation: For arbitrary
$u,v \in V$, $$K(u,v):=R(u,v,u,v).$$ If $u$ and $v$ are orthornormal,
this is the sectional curvature of the 2-plane spanned by $u$ and
$v$. For any $u$, let $$H(u):=R(u,J(u),u,J(u)).$$ If $u$ is a unit
vector this is the holomorphic sectional corresponding to $u$. Note
that $H(u)=-1$ for all unit vectors $u$ if $R=R_0$.

\begin{lem}
Given $\eps >0$, there exists \ $\eta =\eta (\eps)$ such that if
$R$ is a K\"ahler curvature tensor with $ \vert H(u)+1 \vert <
\eta$ for all unit vectors $u$, then $\vert R -R_0 \vert < \eps$.
\end{lem}
\begin{proof}
Note that $\vert R -R_0 \vert < \eps$ if $\vert R(x,y,z,t)
-R_0(x,y,z,t) \vert < \frac {\eps}{(2n)^4}$ for every orthonormal
set $\{x,y,z,t\}$.

Let $u, v \in V$ be unit vectors. If $a,b$ are real numbers such
that $a^2 + b^2=1$, then (this calculation is from ~\cite{kg},
Lemma 3)
\begin{align} \notag
H(au+bv)\ + \ &H(au-bv) \ = \ 2a^4H(u)\ + \ 2b^4H(v) \
\\
\notag & + \ 12a^2b^2R(u,J(u),v,J(v))
 - 8a^2b^2R(u,v,u,v)
\end{align}
\begin{align} \notag
H(au+bJ(v)) \ + \ &H(au-bJ(v)) \ = \ 2a^4H(u) \ + \ 2b^4H(v)
\\ \notag & + \ 12a^2b^2R(u,J(u),v,J(v)) - a^2b^2R(u,J(v),u,J(v))
\notag
\end{align}
  We also have the equation
$$R(u,v,u,v) +R(u,J(v),u,J(v))-R(u,J(u),v,J(v))=0,$$ which follows from (3)
and (4) in the definition of a K\"ahler curvature tensor. Take
$a=b = \frac {1}{\sqrt 2}$. It is clear that we can use these
three equations to express $R(u,v,u,v)$ as a linear combination of
$$H(u), H(v),H(\frac {u+v}{\sqrt 2}), H(\frac {u-v}{\sqrt 2}),
H(\frac {u+J(v)}{\sqrt 2}),H(\frac {u-J(v)}{\sqrt 2}).$$ Moreover note
that the coefficients in this expression are absolute constants, not
dependent on anything.

On the other hand, one knows that the full curvature tensor can be
expressed in terms of sectional curvatures. In fact ~\cite{men},

\begin{align}\notag
24 R(x,y,&z,t)= K(x+z,y+t)+K(x-z,y-t)-K(x+z,y-t)-K(x-z,y+t)
\\ \notag &-K(x+t,y+z)-K(x-t,y-z)+K(x+t,y-z)+K(x-t,y+z) \\
              \notag
\end{align}
It is clear from this expression and the earlier one expressing
$K(x,y)$ in terms of holomorphic sectional curvatures that $ \vert
R(x,y,z,t) -R_0(x,y,z,t) \vert < \frac {\eps}{(2n)^4}$ if $\vert
H(u) +1 \vert < \eta$ for an $\eta=\eta(\eps,n)$.

\end{proof}
\begin{prop}\label{ta}
Given $\eps >0$, there exists $\delta =\delta (\eps,n)$ such that
if the sectional curvatures of a K\"ahler curvature tensor $R \in
\rc$ satisfy
$$ -1 - \delta \ \le  \ K  \ \le \ - \frac {1}{4} $$
then
\begin{equation}\notag
\vert R - R_0 \vert < \eps
\end{equation}
in the norm on $ \otimes^4 V^\ast$.
\end{prop}
\begin{proof}

By the lemma above, it is enough to show that there is a $\delta
=\delta (\eta)
>0$ such that $-1 - \delta \ \le  \ K  \ \le \ - \frac {1}{4}$
implies $ \vert H(u)+1 \vert < \eta$ for all unit vectors $u$.
\vspace{2mm}

Let $u$ be any unit vector in $V$. Choose $v$ so that $\{
u,J(u),v,J(v) \}$ is an orthonormal set. By (3) and (4) in the
definition of a K\"ahler curvature tensor,
\begin{equation} \label{one}
K(u,v) +K(u,J(v))-R(u,J(u),v,J(v))=0.
\end{equation}
We recall the following estimate of Berger (\cite{ber}, Page 69,
(7)). Note that though Berger works with positive pinching, his
proof works for negative pinching as well. Also, our pinching
hypothesis and the conclusion can be obtained from his by a
rescaling. If $R$ is a curvature tensor (not necessarily
K\"ahler) which is $\alpha$-pinched, i.e. $ -\alpha \le K \le -
\frac{1}{4}$ then
$$ \vert R(X,Y,Z,W) \vert \le \frac {2}{3}(\alpha - \frac {1}{4}).$$
for any orthonormal set $\{X,Y,Z,W \}$.

Let $\eta >0$ be as in Lemma \ref{ta}. By Berger's estimate there
exists $\delta_1 = \delta_1 ( \eta)$ such that if the sectional
curvatures $K$ satisfy
\begin{equation}\label{pre}
-1 - \delta_1 \ \le  \ K  \ \le \ - \frac {1}{4}
\end{equation}
then
\begin{equation} \label{two}
\vert R(u,J(u),v,J(v)) \vert \le \frac {1}{2} + \frac {\eta}{6}.
\end{equation}

By (\ref{one}) and (\ref{two})
\begin{align}\notag
K(u,v) \ &= \ - K(u,J(v)) \ + \ R(u,J(u),v,J(v)) \\ \notag
        & \ge \  \frac{1}{4} \ - \ \frac{1}{2} \ -  \frac {\eta}{6} \ = \ - \frac{1}{4} \ -  \frac {\eta}{6}
        \label{three}\\
\end{align}
Of course, we have the same estimates for $K(u,J(v)), \
K(J(u),v))$ and $ K(J(u),J(v))$. Hence,
\begin{equation}\label{four}
-\frac {1}{4} - \frac {\eta}{6} \ \le \ c \ \le \ - \frac {1}{4}, \
\ \ \ \ \ - \frac {1}{2} - \frac {\eta}{3} < R(u,J(u),v,J(v)) \le -
\frac {1}{2},
\end{equation}
where
$$c \in \{K(u,v), \ K(u,J(v)), \ K(J(u),v), \ K(J(u),J(v)) \}.$$
\vspace{2mm}

Applying the second inequality in (\ref{four}) to the orthonormal
set
$$\{ \frac {1}{\sqrt 2}(u -J(v)), \ \frac {1}{\sqrt 2}(J(u) +v),
\ \frac {1}{\sqrt 2}(u +J(v)), \ \frac {1}{\sqrt 2}(J(u) -v)\} ,$$
we get
\begin{equation} \label{five}
 R( u -J(v), \ J(u) +v, \ u +J(v), \ J(u) -v ) \le - 2.
\end{equation}
Expanding the left hand side
\begin{equation}
R( u -J(v), \ J(u) +v, \ u +J(v), \ J(u) -v ) =
\end{equation}
$$H(u) + H(v) +2R(u,J(u),v,J(v))
-4K(u,v)$$
Combining (\ref{four}) and (\ref{five}) we have
\begin{align}
 \ H(u) \ + \ H(v) \ - \ 4K(u,v) \ < \ -1 +  \frac {2\eta}{3}. \label{eh1}
\end{align}

(\ref{eh1}) and the first inequality in (\ref{four}) imply
$$   H(u) \ + \ H(v) \ < \ -2 +  \frac {2\eta}{3}.$$
Finally, using (\ref{pre}), the above inequality implies
$$-1 - \delta_1   \ < \ H(u) \ < \ -1+  \frac {2\eta}{3} + \delta_1 .$$
Let
$$ \delta = min \{ \frac{\eta}{3}, \delta_1 \}.$$
Then $$\vert H(u) +1 \vert \ < \  \eta$$ and the proof is complete
since the unit vector $u$ was arbitrary.



\end{proof}

Let $\omega$ denote the K\"ahler form of $V$, defined by $\omega
(u,v) =\langle u, J(v) \rangle$ Then the volume form on $V$, with
$V$ carrying the natural orientation coming from $J$, is just
$\omega ^n$.

Let $I=(a_1,...,a_n)$ and $J=(b_1,...,b_n)$ be muti-indices in
$({\mathbb Z}_{\ge 0})^n$ with $\Sigma_i ia_i=\Sigma_iib_i=n$.
Given an K\"ahler curvature tensor $R$ we can form the
corresponding Chern forms $c_I(R) = c_1^{a_1}(R) \wedge...\wedge
c_n^{a_n}(R)$ and $c_J(R)=c_1^{b_1}(R) \wedge...\wedge
c_n^{b_n}(R)$. These are elements of $\wedge^n V^\ast$.


Given $\eps_1$ (to be chosen below),  Proposition (\ref{ta}) (and
the obvious continuous dependence of the Chern forms on the full
curvature tensor) implies there is a $\delta = \delta(\eps_1,n)$
such that $-1 -\delta < K < -1$ implies that

\begin{equation} \label{qua}
\vert c_I(R)  - c_I(R_0) \vert < \eps_1, \ \ \
 \vert c_J(R) - c_J(R_0) \vert < \eps_1
\end{equation}
in the norm on $ \wedge^n V^\ast$.

Let $c_I(R_0) = a(n) \omega^n$ and $c_J(R_0) = b(n) \omega^n$.
Choose $\eps_1 = \eps_1(\eps, n)$ so that
\begin{equation}\label{bla}
\frac {a(n) - \eps_1}{b(n) + \eps_1} > \frac {a(n)}{b(n)} - \eps
\ \ \ \ \ \ \frac {a(n) + \eps_1}{b(n) - \eps_1} < \frac
{a(n)}{b(n)} + \eps
\end{equation}

Now (\ref{qua}) implies
$$  (a(n) - \eps_1) \omega^n  \ <  \ c_I(R)  \ < \  (a(n) +\eps_1)
\omega^n$$ and
$$  (b(n) - \eps_1) \omega^n  \ < \ c_J(R) \  <  \ (b(n) +\eps_1)
\omega^n$$ \vspace{3mm}

We now turn to a compact K\"ahler manifold with almost-quarter
pinching as in Theorem 1.1. We can integrate the above
inequalities on $M$ to get
$$  (a(n) - \eps_1) Vol(M)  \ < \ c_I(M) \ <  \ (a(n) +\eps_1) Vol(M) $$
and
$$ (b(n) - \eps_1) Vol(M)  \ <  \ c_J(M) \ <  \ (b(n) +\eps_1) Vol(M).$$
Dividing these two inequalities and using (\ref{bla}) gives the
desired result.


\begin{thebibliography}{10}

\bibitem{ber} M. Berger,
\textit{Sur quelques vari\'{e}t\'{e}s Riemanniennes suffisamment
pinc\'{e}es}, Bull. Soc. Math. France, \textbf{88} (1960), 57--71.

\bibitem{mar} M. Deraux,
\textit{A negatively curved K\"ahler threefold not covered by the
ball}, Invent. Math. \textbf{160} (2005), 501--525.

\bibitem{gt} M. Gromov, W. P. Thurston, \textit{Pinching constants for
hyperbolic manifolds}, Invent. Math. \textbf{89} (1987), 1--12.

\bibitem{her} L. Hernandez, \textit{K\"ahler manifolds and
    $\frac{1}{4}$-pinching}, Duke Math. J. \textbf{62} (1991)
  601--611.

\bibitem{kg} S. Kobayashi, S. I. Goldberg,
\textit{Holomorphic bisectional curvature}, J. Diff. Geom.,
\textbf{1} (1967), 225--233.

\bibitem{men} S. Mendonca, D. Zhou,
\textit{Expression of curvature tensors and some applications},
Bol. Soc. Bras. Mat., \textbf{32} (2001), No. 2, 173--184.

\bibitem{ms} G. D. Mostow, Y. T. Siu,
\textit{A compact K\"ahler surface of negative curvature not
covered by the ball}, Ann. Math \textbf{112} (1980), 321--360.

\bibitem{pan} P. Pansu, \textit{Pincement des vari{\'e}t{\'e}s  {\`a}
courbure n{\'e}gative d'apr{\`e}s M. Gromov et W. Thurston}, S{\'e}minaire
de Th{\'e}orie Spectrale et G{\'e}om{\'e}trie, no.~4 (1985--1986), 101--113.

\bibitem{zy} S. T. Yau, F. Zheng,
\textit{Negatively $\frac{1}{4}$-pinched Riemannian metric on a compact K\"ahler manifold}, 
Invent. Math. \textbf{103} (1991), 527-535. 



\end{thebibliography}
\end{document}